\documentclass[11pt,a4paper,twoside]{amsart}

\usepackage{amssymb} 
\usepackage{amsmath}
\usepackage{amsthm}
\usepackage{dsfont}
\usepackage{color}
\usepackage{graphicx}

\usepackage{verbatim}
\usepackage{url}
\newtheorem{thm}{Theorem}[section]
\newtheorem{lem}{Lemma}[section]

\newtheorem{cor}{Corollary}[section]

\usepackage{caption}
\usepackage{array} 
\newcolumntype{M}[1]{>{\centering\arraybackslash}m{#1}} 
\usepackage{stackengine} 

\def\1{\mathds{1}}

\DeclareMathOperator{\Log}{log}

\newenvironment{proofbold}[1][\proofname]{{\noindent \bfseries #1.}}{\hfill$\square$\newline}

\usepackage[dvipsnames]{xcolor}

\usepackage[colorlinks=true]{hyperref}
\hypersetup{citecolor=blue, linkcolor=blue}

\title[... to explicit estimates for the
  M\"obius function -- II]{From explicit estimates for the primes to explicit estimates for the
  M\"obius function -- II%
}
\author{Olivier Ramar\'e}
\author{Sebastian Zuniga-Alterman}
\thanks{$^*$ The second author has been supported by the Finnish Centre of Excellence in Randomness
and Structures grant 346307 of the Academy of Finland,}
\newcommand{\Addresses}{{
  {\footnotesize
\ \\
  O.~RAMAR\'E - \textsc{CNRS / Institut de Math\'ematiques de Marseille, Aix Marseille Universit\'e,
U.M.R. 7373, Campus de Luminy, Case 907, 13288 MARSEILLE Cedex 9, France.}\\
  \texttt{olivier.ramare@univ-amu.fr}

\ \\
  S.~ZUNIGA ALTERMAN - \textsc{Department of Mathematics and Statistics, University of Turku, 20014 TURKU, Finland.}\\
  \texttt{szualt@utu.fi}

}}}

\begin{document}



\maketitle

\begin{abstract}
  We improve on all the results of \cite{Ramare*12-2} by incorporating the
  finite range computations performed since then by several
  authors. Thus we have
  \begin{align*}
    \Bigg|\sum_{n\le X}\mu(n)\Bigg|
    &\le \frac{0.006688\,X}{\Log X},&&\text{for } X\ge 1\,798\,118,\\
    \Bigg|\sum_{n\le X}\frac{\mu(n)}{n}\Bigg|
   & \le \frac{0.010032}{\Log X},&& \text{for } X\ge 617\,990.
  \end{align*}  
  We also improve on the method described in \cite{Ramare*12-2} by a simple remark. 
\end{abstract}


\section{Introduction and results}
Let $\mu$ the M\"obius function and $\Lambda$ the von Mangoldt function. In \cite{Ramare*12-2}, the first author exploited the identity
\begin{equation}
  \label{enlighteningbis}
  2\gamma+\sum_{n\le X}\mu(n)\Log^2n
  =\sum_{k\ell \le X}\mu(\ell)\bigl(\Lambda\star\Lambda(k)-\Lambda(k)\Log k+2\gamma\bigr),
\end{equation}
valid for any $X\geq 1$, to derive explicit estimates for the summatory function $M$ of the
M\"obius function $\mu$. In this article, we propose to update those estimates by
taking into account \cite{Hurst*18} by G.~Hurst,
\cite{Buthe*18} by J.~B\"uthe,
\cite{Vanlalnagaia*15-1} by R.~Vanlalnagaia and \cite{Ramare*12-1} by the first author.
We also improve on their corresponding proofs by essentially taking a closer look
at them, see for instance Lemma~\ref{majR2}.  Moreover, we hope to have improved
on the exposition with respect to \cite{Ramare*12-2}. In order to do so, we shall reproduce some lemmas whose proofs can still be found in
\cite{Ramare*12-2}, so that the reader may follow the
argument more easily. By the same reason, more on the history of this problem and on the philosophy of the
method we use can be found in the original work~\cite{Ramare*12-2}.

Below we state our results, which should be directly compared to those presented in \cite{Ramare*12-2} in their corresponding ranges.
\begin{thm}
  \label{mainbound}
  For $X\ge 1\,798\,118$, we have
  \begin{equation*}
    \Bigg|\sum_{n\le X}\mu(n)\Bigg|
    \le \frac{(0.006688\,\Log X-0.039)X}{\Log^{2}X}.
  \end{equation*}
\end{thm}
This improves  by almost a factor of $2$ on \cite[Thm. 1.1]{Ramare*12-2}, which we  recover as
the next corollary by adding some simple computations to lower the range of~$X$.
\begin{cor}
  \label{mainbound2}
  For $X\ge 1\,078\,853$, we have
  \begin{equation*}
    \Bigg|\sum_{n\le X}\mu(n)\Bigg|
    \le \frac{(0.0130\Log X-0.118)X}{\Log ^2X}.
  \end{equation*}
\end{cor}
On the other hand, by the method outlined in $\S$\ref{strategy}, we also derive a bound on the logarithmic average of the M\"obius function.
\begin{cor}
  \label{mainboundbis}
  For $X\ge 617\,990$, we have
  \begin{equation*}
    \Bigg|\sum_{n\le X}\frac{\mu(n)}{n}\Bigg|
    \le      \frac{0.010032\log X-0.0568}{\log^2 X}.
  \end{equation*}
\end{cor}
As a consequence, we correctly obtain \cite[Thm. 1.2]{Ramare*12-5}.
\begin{cor}
  \label{mainboundbis2}
  For $X\ge 463\, 421$, we have
  \begin{equation*}
    \Bigg|\sum_{n\le X}\frac{\mu(n)}{n}\Bigg|
    \le \frac{0.0144\Log X-0.1}{\Log^2X}.
  \end{equation*}
\end{cor}

\subsubsection*{\textbf{Notation}}
We set
\begin{equation}\label{rR}
R(X)=\psi(X)-X,\qquad\qquad r(X)=\tilde\psi(X)-\Log
X+\gamma,
\end{equation} 
where, by denoting $\Lambda$ the von Mangoldt function, 
\begin{equation}
  \label{defpsi}
  \psi(X)=\sum_{n\leq X}\Lambda(n),\qquad\qquad\tilde{\psi}(X)=\sum_{n \le X}\frac{\Lambda(n)}{n}.
\end{equation}
On the other hand, we consider the following summatory functions related to the M\"obius function
\begin{equation}\label{mM}
  M(X)=\sum_{n\le X}\mu(n),\qquad\qquad
  m(X)=\sum_{n\leq X}\frac{\mu(n)}{n}.
\end{equation}
\section{New estimates and consequences}

Let us turn our attention to the summatory function of the M\"obius function $M$. From \cite{Hurst*18}, we have the following estimation
\begin{equation}
  \label{boundsM}
  |M(X)|\le
    \sqrt{X},\qquad X\le 10^{16}.
\end{equation}
In \cite{Dress*93}, we find the bound
\begin{equation}
  \label{eq:13}
  |M(X)|\le 0.571\sqrt{X},\qquad 33\le X\le 10^{12}
\end{equation}
In \cite{Dress-ElMarraki*93}, we find that
\begin{equation}
  \label{eq:16}
  |M(X)|\le \frac{X}{2360},\qquad X\ge 617\,973,
\end{equation}
(see also \cite{CostaPereira*89}) which \cite{Cohen-Dress-ElMarraki*96}
(also published in \cite[Thm. 5 bis]{Cohen-Dress-ElMarraki*07}) improves as
\begin{equation}
  \label{eq:17}
  |M(X)|\le \frac{X}{4345},\qquad X\ge 2\,160\,535.
\end{equation}

We also present two bounds concerning the squarefree numbers. The
first one comes from \cite[Thm. 3 bis]{Cohen-Dress-ElMarraki*07}.
\begin{lem}\label{sqf}
  For $X\ge438\,653$, we have
  \begin{equation*}
    \sum_{n\le X}\mu^2(n)=\frac{6}{\pi^2}X+O^*(0.02767\sqrt{X}).
  \end{equation*}
  If $X\ge1$, we can replace $0.02767$ by $0.7$.
\end{lem}
\begin{lem}\label{sqflog}
  We have for $X\ge 10^6$
  \begin{equation*}
    \sum_{n\le X}\frac{\mu^2(n)}{n}=\frac{6}{\pi^2}\log X+O^*(1.045).
  \end{equation*}
  If $X\ge1$, we can replace $1.045$ by $1.48$.
\end{lem}
\begin{proof} It is enough to quote \cite[Cor. 1.2]{Ramare*18-9}, so that
  \begin{equation*}
    \sum_{n\le X}\frac{\mu^2(n)}{n}=\frac{6}{\pi^2}(\log X+\mathfrak{b})+O^*\bigg(\frac{0.43}{\sqrt{X}}\bigg),
  \end{equation*}
  where $\mathfrak{b}=1.7171\cdots$.
\end{proof}
With respect to $m$, defined in \eqref{mM}, we find in \cite[Lemma 5.10]{Helfgott19} the following result
 \begin{equation}
  \label{boundsm}
  |m(X)|\le
    \frac{\sqrt{2}}{\sqrt{X}} \quad \text{when $X\le 10^{14}$}.
\end{equation}

On the other hand, recall definition \eqref{rR}. We have the following set of estimations for the function $R$ in different ranges.
\begin{align}
  \label{boundsR1}
      |R(X)|&\le 0.8\sqrt{X}&&\text{when $1500< X\le 10^{10}$},\\
    \label{boundsR2}
       |R(X)|&\le 0.94\sqrt{X}&&\text{when $11<X\le 10^{19}$},\\
   \label{boundsR3}
       |R(X)|&\le  8\cdot 10^{-5}\cdot X&&\text{when $10^8\le  X$},\\
    \label{boundsR4}
       |R(X)|&\le 2.58843\cdot 10^{-5}\cdot X&&\text{when $21\le \log X$},\\
    \label{boundsR5}
       |R(X)|&\le 1.93378\cdot 10^{-8}\cdot X&&\text{when $4
      0\le \log X$},\\
       \label{boundsR7}
       |R(X)|&\le 0.0065 \frac{X}{\Log X}&&\text{when $X\ge1\,514\,928$}\\
    \label{boundsR6}
       |R(X)|&\le 9\cdot 10^{-7}\frac{X}{\log X}&&\text{when $10^{19}\le X$},
\end{align}
Estimate \eqref{boundsR1} comes from \cite[p. 423]{Ramare-Rumely*96}.
Estimate \eqref{boundsR2} comes from \cite[Thm. 2]{Buthe*18}.
\eqref{boundsR4} and \eqref{boundsR5} are derived from \cite[Table
8]{Broadbent-Kadiri-Lumley-Ng-Wilk*21} together with \cite[Table 1]{Buthe*16} .
With respect to \eqref{boundsR3}, the same table gives the bound $4.27\cdot 10^{-5}$
provided that $X\ge e^{20}$, and we extend it using \eqref{boundsR1}. \eqref{boundsR7} can be found in \cite[Lemma 4.2]{Ramare*12-2} and the last estimate \eqref{boundsR6} comes from \cite[Table
15]{Broadbent-Kadiri-Lumley-Ng-Wilk*21} extended to $\psi$ via
\cite[Thm. 6 (5.3)-(5.4)]{Rosser-Schoenfeld*75}.

Let us further recall the second part of \cite[Theorem 9]{Vanlalnagaia*15-1}.
\begin{equation}
  \label{boundr}
  \bigg|r(X)-\frac{R(X)}{X}\bigg|\le
    \frac{0.05}{\sqrt{X}}
    +1.75\cdot 10^{-12}\quad\text{when $X\geq 394\,385$}.
  \end{equation}

In this article, by splitting into sufficiently small intervals, we propose the following bound
\begin{equation}
  \label{boundRbis}
  \max_{24\,200\le X\le 3\cdot 10^7}\frac{|R(X)|}{\sqrt{X}}\le 0.71.
\end{equation}
We also recall \cite[Thm. 12]{Rosser-Schoenfeld*62}.
\begin{lem}\label{R-S} The quotient $\psi(X)/X$ takes its maximum at $X=113$ and, for any $X>0$,
\begin{equation*}
\psi(X)<1.03883 X.
\end{equation*}
\end{lem}
Finally, we quote \cite[Lemma 5.1]{Ramare*12-2}.
\begin{lem}
  \label{auxb}
  When $X\ge1$, and $\sqrt{X}\ge T\ge 1$, we have
  \begin{equation*}
    \sum_{n\le T}\frac{\Lambda(n)}{n\Log(X/n)}
    \le 1.04\ \Log\bigg(\frac{\Log X}{\Log(X/T)}\bigg)+\frac{1.04}{\Log X}.
  \end{equation*}
\end{lem}

\section{Strategy of proof}\label{strategy}

Call L the function $t>0\mapsto\log t$ and $\star$ the convolution of any two arithmetic functions. By observing that $(1/\zeta)''=-\zeta''/\zeta^2+(\zeta')^2/\zeta^3$, where $\zeta$ corresponds to the Riemann zeta function, we can deduce that 
\begin{equation*}
\mu\cdot\text{L}^2=\mu\star\Lambda\star\Lambda-\mu\star\Lambda\cdot\text{L},
\end{equation*}
whence identity \eqref{enlighteningbis}.

We define the remainder quantity~$R_2^*$ as follows
\begin{equation}
  \label{defR2}
 R_2^*(X)= \sum_{n\le X}\left(\Lambda\star\Lambda(n)-\Lambda(n)\Log n+2\gamma\right).
\end{equation}
Now, for any $K\in(0,X]\cap\mathbb{Z}$, we may derive from \eqref{enlighteningbis} the following expression
\begin{align}\label{hyp}
\sum_{n\leq X}\mu(n)\log^2n&=-2\gamma+\sum_{\ell\leq X/K}\mu(\ell)R_2^*\bigg(\frac{X}{\ell}\bigg)\\
&\qquad\qquad+\sum_{k<K}R_2^*(k)\sum_{X/(k+1)<\ell\leq X/k}\mu(\ell).\nonumber
\end{align}
Observe that the range in the outermost sum of the above second expression is in fact $\max\{X/K,X/(k+1)\}<\ell\leq X/k$. Nonetheless, if $X/(k+1)<X/K$, since $k<K$, we have $k<K<k+1$, giving an empty sum as $K$ is integer.

Moreover, by rearranging terms, we may express
\begin{align}\label{exp}
\sum_{k\leq K-1}R_2^*(k)\sum_{X/(k+1)<\ell\leq X/k}\mu(\ell)=\sum_{k\leq K-1}R_2^*(k)\bigg[M\bigg(\frac{X}{k}\bigg)-M\bigg(\frac{X}{k+1}\bigg)\bigg]\nonumber\qquad&\\
=\sum_{k\leq K}(\Lambda\star\Lambda(k)-\Lambda(k)\log k+2\gamma)M\bigg(\frac{X}{k}\bigg)- R_2^*(K)M\bigg(\frac{X}{K}\bigg).&
\end{align} 
Thus, by combining \eqref{hyp} and \eqref{exp} together, we arrive at
\begin{align}\label{light*}
\sum_{n\leq X}\mu(n)\log^2n=-2\gamma+\sum_{\ell\leq X/K}\mu(\ell)R_2^*\bigg(\frac{X}{\ell}\bigg)\quad\qquad\qquad\qquad\qquad&\\
+\sum_{k\leq K}(\Lambda\star\Lambda(k)-\Lambda(k)\log k+2\gamma)M\bigg(\frac{X}{k}\bigg)- R_2^*(K)M\bigg(\frac{X}{K}\bigg).&\nonumber
\end{align}
In order to prove Lemma \ref{victoire}, we use will identity \eqref{light*}, and for this reason, we will require estimates for
$R_2^*(X)$ when $X\ge 1.8\cdot 
10^9$ as well as Lemma \ref{aux1}. 

Concerning the estimation of $R_2^*$, as per Lemma \ref{majR2}, it is enough to estimate two auxiliary functions
$R_3$ and $R_4$. First, $R_3$ is defined as 
  \begin{equation}
    \label{defR3}
    R_3(X)=2\sqrt{X}|\sqrt{X}r(\sqrt{X})-
    R(\sqrt{X})|
    +R(\sqrt{X})^2+|R(X)|\Log X
  +\left|\int_{1}^XR(t)\frac{dt}{t}\right|,
  \end{equation}
  and will be studied in $\S$\ref{bR3}, whereas $R_4$ is defined as
\begin{equation}
  \label{defR4} 
  R_4(X)=\sum_{n\leq\sqrt{X}}\Lambda(n)R\bigg(\frac{X}{n}\bigg),
\end{equation}
and will be studied in $\S$\ref{bR4}.

Thus, upon estimating the right-hand side of \eqref{light*}, we will have an estimation for the sum $\sum_{n\leq X}\mu(n)\log^2n$ within some specific range. This main bound is expressed in Lemma~\ref{victoire} and from it, we
may deduce estimates for $M$ via Lemma \ref{summ}. Indeed, let $(f(n))_{n\in\mathbb{Z}_{>0}}$ be a sequence of complex numbers and, for any integer
  $k\ge0$ and $X\geq 1$, consider the weighted summatory function
\begin{equation}
  \label{eq:26}
  M_k(f, X) = \sum_{n\le X} f(n) \Log^k n .
\end{equation}

Then, by summation by parts or quoting  \cite[Lemma 8.1]{Ramare*12-2}, we have the following result. 
\begin{lem}\label{summ}
  For any $X_0>1$ and any $X\ge X_0$, we have
  \begin{equation*}
    M_0(f, X)-M_0(f, X_0)
    =
    \frac{ M_k(f, X) }{ \Log^{k}X }
    - \frac{ M_k(f, X_0) }{ \Log^{k}X_0 }
    +k\int_{X_0}^X \frac{ M_k(f, t) }{ t\Log^{k+1}(t) } dt.
  \end{equation*}
\end{lem}
Finally, upon having an estimation for $M(X)$ within some range, we are
able to derive bounds for $m(X)$ by using the following result, which may be found in \cite[Thm. 1.1]{Ramare*12-5}.
\begin{lem}\label{EM}[Balazard]
  For any $X\geq 1$, we have
  \begin{equation*}
    |m(X)|\le \frac{|M(X)|}{X}
    +\frac{1}{X^2}\int_1^X|M(t)|dt
    +\frac{8}{3X}.
  \end{equation*}
\end{lem}
 At
the time of writing, the best produced outcome seems to be the one given by F. Daval
\cite{Daval*24}. 

\section{Bounding $R_3$}\label{bR3}
Let us recall \cite[Lemma 6.4]{Ramare*12-2}.
\begin{lem}\label{int1}
  \begin{equation*}
    \int_1^{10^8}R(t)\frac{dt}{t}=-129.559+O^*(0.01).
  \end{equation*}
\end{lem}
We can next give an estimation for $R_3$, defined in \eqref{defR3}.
\begin{lem} We have
  \label{boundR3}
  \begin{align*}
  R_3(X)&\le 0.2\cdot X^{3/4},&&\qquad\text{ when
                                                }1.8\cdot 10^9\le X\le 10^{19},\\
  \frac{R_3(X)}{X}&\le 9\cdot 10^{-5}+\frac{1}{10 X^{1/4}},&&\qquad\text{ when }X\ge 10^{19}.
  \end{align*}
\end{lem}

\begin{proof}
We consider three cases. Let $K_0=1.8\cdot 10^9$ and $K=10^8$.

\smallskip\noindent{\sl Case 1: suppose that $1.8\cdot 10^{9} \le X\le 10^{10}$}.
 By Lemma \ref{int1}, we have that
  \begin{equation*}
  \left|\int_{1}^XR(t)\frac{dt}{t}\right|\leq 129.7+\int_{10^8}^X|R(t)|\frac{dt}{t}.
  \end{equation*}
  On the other hand, we find that
  \begin{align}
 \label{p1} 2\sqrt{X}|\sqrt{X}r(\sqrt{X})-R(\sqrt{X})|&\leq 0.1\cdot X^{3/4}+3.5\cdot 10^{-12}X ,\\
 \label{p2} R(\sqrt{X})^2&\leq 0.8^2\sqrt{X},\\
  \label{p3}|R(X)|\Log X&\leq 0.8\sqrt{X}\Log X,\\
  \label{p4}\int_{10^8}^X|R(t)|\frac{dt}{t}&\leq 0.8\int_{10^8}^X\frac{dt}{\sqrt{t}}=2\cdot0.8(\sqrt{X}-10^4).
  \end{align}
  Inequality~\eqref{p1} is obtained thanks to \eqref{boundr} whereas
  \eqref{p2}, \eqref{p3} and \eqref{p4} come from estimation
  \eqref{boundsR1}.
    Therefore, by combining the above bounds together, we arrive at
    \begin{align}\label{BB}
    \frac{R_3(X)}{X^{3/4}}
    \le
    &0.1+3.5\cdot 10^{-12}X^{1/4}+0.8\frac{\log X}{X^{1/4}}+
    \nonumber\\
    &\qquad+\frac{0.8^2+2\cdot 0.8}{X^{1/4}}
    + \frac{129.7-2\cdot 0.8\cdot10^4}{X^{3/4}}\quad \le 0.2,
    \end{align}
  where we have used that the function $t\mapsto\log t/t^{1/4}$ is
  decreasing for $\log t\geq 4$. 
  
  \smallskip\noindent{\sl Case 2: suppose that $10^{10} \le X\le 10^{19}$}. Now the bounds \eqref{p2}, \eqref{p3}, \eqref{p4} and \eqref{BB} hold with $0.8$ replaced by $0.94$. Thus we obtain the estimation
        \begin{equation*}
    \frac{R_3(X)}{X^{3/4}}\le 0.18.
    \end{equation*}

\smallskip\noindent{\sl Case 3: suppose that  $X\ge 10^{19}$}. We proceed similarly as
  before and obtain
  \begin{align*}
    \frac{R_3(X)}{X}
    &\le
    3.5\cdot 10^{-12}
    + \frac{0.1}{X^{1/4}}+ (2.59\cdot 10^{-5})^2+9\cdot 10^{-7}\\
    &\qquad\qquad+8\cdot 10^{-5}-\frac{8\cdot 10^3-129.7}{X}\\
    &\le 0.000081+ \frac{0.1}{X^{1/4}}\quad\le 0.000083,
  \end{align*}
  where we have used \eqref{boundsR4} instead of \eqref{boundsR2}, since  $e^{21}\leq 10^{19/2}$, then \eqref{boundsR6} instead of \eqref{boundsR2} and lastly \eqref{boundsR3} instead of \eqref{boundsR2}.
  \end{proof}

\section{Bounding $R_4$}\label{bR4}
\begin{lem}\label{compR4}
  When $1.8\cdot 10^9\le X\le 5\cdot 10^{10}$, we have
  \begin{equation*}
    \left|\sum_{1000<n\le 40\,000}\Lambda(n)R\bigg(\frac{X}{n}\bigg)\right|\le
    0.000154\cdot X.
  \end{equation*}
\end{lem}
Whereas we have run the computations up to $5\cdot10^{10}$, we will only use
the above result up to $2\cdot 10^{10}$. Employing the larger range might improve on our final bound, but we have opted to keep the latter shorter one to ensure accuracy.
\begin{proof}
  Let us set, for any two integers $A_0$ and $A_1$:
  \begin{equation}
    \label{deffAD}
    f(A_0,A_1,D)=\sum_{A_0\le a\le A_1}\Lambda(a)\psi\bigg(\frac{X}{a}\bigg).
  \end{equation}
  We compute and store the values of $\psi(n)$ for $n\in[10^9/A_1,\cdots,10^{10}/A_0]\cap\mathbb{Z}$.
  
  Next, as $X$ moves from the integer $N$ to $N+1$, we may get a modification in the value of $\psi(X/a)$ only if
  $[N/a]<[(N+1)/a]$. This may only happen when there exists an
  integer $c$ such that $N/a< c\le (N+1)/a$, that is such that $N< ac \le N+1$,
  so only whenever $a$ is a divisor of $N+1$, which must further belong to
  the interval $[A_0,A]$. This observation speeds up the program we use considerably and
  makes the computations possible on a home computer, using merely a couple of days.
\end{proof}
Recall now definition \eqref{defR4}.
\begin{lem}\label{auxc}
  We have 
  \begin{equation*}
  |R_4(X)|\le0.005\cdot X,\qquad X\ge 1.8\cdot 10^9.
  \end{equation*}
\end{lem}
This improves slightly on \cite[Lemma 5.2]{Ramare*12-2}.

\begin{proof}
We consider different cases according to the range of $X$.

\smallskip\noindent{\sl Case 1: suppose that $1.8\cdot 10^9 \le X\le 9\cdot 10^{9}$}. We first check that
  \begin{equation}
    \label{eq:1}
    \sum_{n\leq 1000}\frac{\Lambda(n)}{\sqrt{n}}\le 60.51,\qquad
    \sum_{n\leq 40000}\frac{\Lambda(n)}{\sqrt{n}}=B= 40012.8937\cdots.
  \end{equation}
  Next, observe that, when $n\leq 1000$, $X/n\in(1500,10^{10})$ and when $40000<n\leq \sqrt{X}$, $X/n\in[\sqrt{X},X/40000]\subset[24200,3\cdot 10^7]$. Thus, by using estimations \eqref{boundsR1}, \eqref{boundRbis} and also Lemma~\ref{compR4}, we have the following estimation
  \begin{equation*}
    |R_4(X)|
    \le 0.8 \sqrt{X}\sum_{n\leq 1000}\frac{\Lambda(n)}{\sqrt{n}}
      +0.000154 \cdot X 
      +0.71\sqrt{X}\sum_{40000<n\le \sqrt{X}}\frac{\Lambda(n)}{\sqrt{n}},
      \end{equation*}
      which, by recalling \eqref{eq:1}, Lemma \ref{R-S} and using summation by parts, may be expressed as
      \begin{align*}
    |R_4(X)|&\le
    48.408\sqrt{X}
    + 0.000154\cdot X
    \\&\qquad+
    0.71 \sqrt{X}\biggl(
      \int_{40000}^{\sqrt{X}}\frac{\psi(t)-B}{2t^{3/2}}dt
    +\frac{\psi(\sqrt{X})-B}{X^{1/4}}
    \biggr)
    \\&\le
    48.408\sqrt{X}
    + 0.000154\cdot X
    \\&\qquad+
    0.71 \sqrt{X}\biggl(
    \int_{40000}^{\sqrt{X}}\frac{1.04\cdot t}{2t^{3/2}}dt
    -\frac{B}{\sqrt{40000}}
    +\frac{1.04\sqrt{X}}{X^{1/4}}
    \biggr)
    \\&\le
    48.408\sqrt{X}
    + 0.000154\cdot X
    \\&\qquad+
    0.71 \sqrt{X}\biggl(
    1.04\cdot X^{1/4}-1.04\cdot \sqrt{40000}
    -\frac{40012.89}{\sqrt{40000}}
    +\frac{1.04\sqrt{X}}{X^{1/4}}
    \biggr).
    \end{align*}
    Therefore,
    \begin{equation*}
      \frac{|R_4(X)|}{X}\le
      \frac{48.408}{\sqrt{X}}
      + 0.000154
      +
    0.71\biggl(
    \frac{2.08}{X^{1/4}}-\frac{408}{X^{1/2}}
    \biggr),
  \end{equation*}
  which is not more than 0.0032. Here, we have used that the function $t\mapsto at^{-1/4}-bt^{-1/2}$ has a maximum at $t=(2b/a)^4$, taking value $a^2/4b$.

  \smallskip\noindent{\sl Case 2: suppose that $9\cdot 10^9 \le X\le 2\cdot 10^{10}$}. Let $X_1=2\cdot 10^7$. Observe first that for any $T$ such that $\sqrt{X}\leq U$, we can write 
    \begin{equation}\label{partition}
  |R_4(X)|\leq
  \sum_{2\leq n\leq X/U}\Lambda(n)\bigg|R\bigg(\frac{X}{n}\bigg)\bigg|+\sum_{X/U<n\leq\sqrt{X}}\Lambda(n)\bigg|R\bigg(\frac{X}{n}\bigg)\bigg|.
\end{equation}
  Select now $U=X_1$. Thereupon, for the first sum above, we have that $X/n\in[X_1,10^{10}]$, thus we may use estimation \eqref{boundsR1}. Concerning the above second sum, we have that $X/n\in[\sqrt{X},X_1]\subset[24200,3\cdot 10^7]$ so that we may use
  \eqref{boundRbis}. This way, considering definition \eqref{defpsi}, using summation by parts and then recalling Lemma \ref{R-S}, we have
  \begin{align*}
    |R_4(X)|
    \le&\ 0.8 \sqrt{X}\sum_{n\le X/X_1}\frac{\Lambda(n)}{\sqrt{n}}
      +0.71\sqrt{X}\sum_{X/X_1<n\le \sqrt{X}}\frac{\Lambda(n)}{\sqrt{n}}
    \\
=&\ 0.71 \sqrt{X}\biggl(
    \int_2^{\sqrt{X}}\frac{\psi(t)}{2t^{3/2}}dt
    +\frac{\psi(\sqrt{X})}{X^{1/4}}
      \biggr)
    \\&\qquad
    +0.09 \sqrt{X}\biggl(
    \int_2^{X/X_1}\frac{\psi(t)}{2t^{3/2}}dt
    +\frac{\psi(X/X_1)}{(X/X_1)^{1/2}}
      \biggr)
    \\
    \le&\ 2\cdot 0.71\cdot 1.04 X^{3/4}
    +2\cdot 0.09\cdot 1.04 \frac{X}{\sqrt{X_1}},
  \end{align*}
  which is not more than $0.0049\cdot X$. 

   \smallskip\noindent{\sl Case 3: suppose that $2\cdot 10^{10} \le X\le 2\cdot 10^{19}$}. Set $X_2=10^{10}$ and let us use expression \eqref{partition} with $U=X_2$. For the arising first sum, we have $X/n\in [X_2,10^{19}]$, so that we may use estimation \eqref{boundsR2}. As for the second one, we have $X/n\in[\sqrt{X},X_2]\subset[1500,10^{10}]$ and we may use estimation \eqref{boundsR1}. Thus, by using summation by parts and recalling Lemma \ref{R-S}, we derive 
\begin{align*}
    |R_4(X)|
    \le&\ 0.94 X\sum_{n\le X/X_2}\frac{\Lambda(n)}{\sqrt{n}}+
      0.8 \sqrt{X}\sum_{X/X_2< n\le \sqrt{X}}\frac{\Lambda(n)}{\sqrt{n}}
    \\ =&\ 0.8 \sqrt{X}\biggl(
    \int_{2}^{\sqrt{X}}\frac{\psi(t)}{2t^{3/2}}dt
    +\frac{\psi(\sqrt{X})}{X^{1/4}}
      \biggr)
      \\&\qquad+0.14
      \sqrt{X}\biggl(
    \int_{2}^{X/X_2}\frac{\psi(t)}{2t^{3/2}}dt
    +\frac{\psi(X/X_2)}{(X/X_2)^{1/2}}
    \biggr)
    \\\le &\ 2\cdot 0.8\cdot 1.04\cdot X^{3/4}
   +2\cdot 0.14\cdot 1.04\frac{X}{\sqrt{X_2}},
  \end{align*}
  which is not more than $0.0045\cdot X$. 

  \smallskip\noindent{\sl Case 4: suppose that $X\ge 2\cdot 10^{19}$.}
  Since $10^{19}>1\,514\,928^2$, by estimation \eqref{boundsR7} and Lemma~\ref{auxb} with $T=\sqrt{X}$, we have that
  \begin{equation*}
    \frac{|R_4(X)|}{X}\le
    0.0065\sum_{n\le \sqrt{X}}\frac{\Lambda(n)}{n\Log (X/n)}
    \le0.0065\cdot\left(0.73+\frac{1.04}{\Log X}\right).
  \end{equation*}
  which implies that $|R_4(X)|\le 0.0049\cdot X$.
  
  Finally, on combining cases $1$, $2$, $3$, $4$ together, we derive the result.
     \end{proof}

\section{Bounding $R_2^*$}\label{bR2*}

Recall definition \eqref{defR2}. 
\begin{lem}
  \label{majR2}
  For any $X>0$, we have 
  \begin{equation*}
  |R_2^*(X)|\le 1+2\gamma+R_3(X)+2R_4(X),
  \end{equation*} 
  where $R_3$ and
  $R_4$ are respectively defined in~\eqref{defR3} and~\eqref{defR4}.
\end{lem}
The reader should compare the above lemma with \cite[Lemma 6.2]{Ramare*12-2}. The novelty here is that 
$|\sqrt{X}r(\sqrt{X})-R(\sqrt{X})|$ appears instead of $|r(\sqrt{X})|$, which is quantity that is well
controlled by~\eqref{boundr}. 
\begin{proof}
  By summation by parts, and recalling definition \eqref{rR}, we have
  \begin{align*}
    \sum_{n\le X}\Lambda(n)\Log n
    &=
    \psi(X)\Log X-\int_{1}^X\frac{\psi(t)}{t}dt
    \\&=X\Log X -X +1+R(X)\Log X-\int_{1}^X\frac{R(t)}{t}dt.
  \end{align*}
  On the other hand, Dirichlet hyperbola formula yields
  \begin{align*}
    &\qquad\sum_{d_1d_2\le X}\Lambda(d_1)\Lambda(d_2)
    =2\sum_{d_1\le \sqrt{X}}\Lambda(d_1)
    \psi\bigg(\frac{X}{d_1}\bigg)
    -\psi(\sqrt{X})^2
    \\=&
    2X\sum_{d_1\le \sqrt{X}}\frac{\Lambda(d_1)}{d_1}
    -X
    -2\sqrt{X}R(\sqrt{X})-R(\sqrt{X})^2+2\sum_{d_1\le \sqrt{X}}\Lambda(d_1)R\bigg(\frac{X}{d_1}\bigg)
    \\=&
    X\Log X-2X\gamma-X
    +2Xr(\sqrt{X})
    -2\sqrt{X}R(\sqrt{X})-R(\sqrt{X})^2+2R_4(X).
  \end{align*}
Hence, by recalling definition \ref{defR2} and combining the above two expressions, we obtain
\begin{align*}
  &R^*_2(X)=\sum_{n\le X}\left(\Lambda\star\Lambda(n)-\Lambda(n)\Log n\right)
  +2[X]\gamma=-1-2\{X\}\gamma+2R_4(X)\\
    &+\bigg[2\sqrt{X}(\sqrt{X}r(\sqrt{X})
    -R(\sqrt{X}))-R(\sqrt{X})^2-R(X)\Log X+\int_{1}^X\frac{R(t)}{t}dt\bigg],
\end{align*}
whence, by recalling definition \eqref{defR3}, we obtain the result.
\end{proof}

Let us recall \cite[Lemma 6.3]{Ramare*12-2}.
\begin{lem}\label{compR2}
  Let $X$ be a real number such that $3\le X\le 2.1\cdot 10^{10}$. Then
  \begin{equation*}
  |R_2^*(X)|\le 1.93\sqrt{X}\Log X.
  \end{equation*}
\end{lem}
Furthermore, the study carried out in \S\ref{bR3}, \S\ref{bR4} and \S\ref{bR2*} allows us to derive the following result.
\begin{lem}
  \label{BoundR2}
  When $X\ge 1.8\cdot 10^9$, we have 
  \begin{equation*}|R_2^*(X)|\le 0.011\cdot X.
  \end{equation*}
\end{lem}
\begin{proof}
  Combine Lemma~\ref{majR2} together with Lemma~\ref{boundR3} and Lemma~\ref{auxc}.
\end{proof}

\section{Proof of Theorem~\ref{mainbound}}
\begin{lem}\label{aux1} Let $X\geq 10^{16}$ and $K>0$ such that $K\cdot 2\,160\,535\leq 10^{16}$. Then
  we have
  \begin{align*}
 &\frac{1}{X}\left| \sum_{k\leq K}(\Lambda\star\Lambda(k)-\Lambda(k)\log k+2\gamma)M\bigg(\frac{X}{k}\bigg)- R_2^*(K)M\bigg(\frac{X}{K}\bigg)\right|\\
 &\qquad\qquad\qquad\le
    \begin{cases}
    0.0374&\text{when $K=462\,848$},\\  
    0.0422&\text{when $K= 10^6$},\\  
    0.0579&\text{when $K=10^7$},\\  
    0.0762&\text{when $K=10^8$}. 
    \end{cases}
  \end{align*}
\end{lem}
\begin{proof} In order to bound the above quantity, we have that $X/k\geq X/K\geq 2\,160\,535$, so that estimation \eqref{eq:17} may be applied. Thereupon, it is enough to compute numerically the following bounds
\begin{equation*}
 \sum_{k\leq K}\frac{\left|\Lambda\star\Lambda(k)-\Lambda(k)\log k+2\gamma\right|}{k}+\frac{|R_2^*(K)|}{K}
 \le
   4345\cdot \begin{cases}
    0.0374,\text{ if $K=462\,848$},&\\  
    0.0422,\text{ if $K= 10^6$},&\\  
    0.0579,\text{ if $K= 10^7$},&\\  
    0.0762,\text{ if $K= 10^8$}.&
    \end{cases}
  \end{equation*}
\end{proof}

\begin{lem}
  \label{aux2}
  Let $X,K,K_0$ be real numbers such that $0<K<K_0\leq X$. When $X/K_0\ge 5\cdot 10^6$ and $\log(K_0/K)\leq 19\,000$, we have
  \begin{equation*}
    \sum_{X/K_0<n\le X/K}\frac{\mu^2(n)}{\sqrt{n}}
    \le
    \frac{12}{\pi^2}\sqrt{\frac{X}{K}}.
  \end{equation*}
\end{lem}

\begin{proof}
  With the help of summation by parts, we find that
  \begin{align*}
    \sum_{X/K_0<n\le X/K}\frac{\mu^2(n)}{\sqrt{n}}
    =
      \int_{X/K_0}^{X/K}
    \sum_{X/K_0<n\le t}\mu^2(n)\frac{dt}{2t^{3/2}}
      +\sum_{X/K_0<n\le X/K}\frac{\mu^2(n)}{\sqrt{X/K}}&
    \\=
    \int_{X/K_0}^{X/K}
    \sum_{n\le t}\mu^2(n)\frac{dt}{2t^{3/2}}
    +\sum_{n\le X/K}\frac{\mu^2(n)}{\sqrt{X/K}}
    -\sum_{n\le X/K_0}\frac{\mu^2(n)}{\sqrt{X/K_0}}&
     \\\le
    \int_{X/K_0}^{X/K}
    \bigg(\frac{6}{\pi^2}t+0.02767\sqrt{t}\bigg)\frac{dt}{2t^{3/2}}
    +\frac{6}{\pi^2}\sqrt{\frac{X}{K}}-\frac{6}{\pi^2}\sqrt{\frac{X}{K_0}}+2\cdot 0.02767&
    \\\le\ 
    \frac{12}{\pi^2}\sqrt{\frac{X}{K}}
    -    \frac{12}{\pi^2}\sqrt{\frac{X}{K_0}}
    +\frac{0.02767}{2}\log\bigg(\frac{K_0}{K}\bigg)+2\cdot 0.02767,&
  \end{align*}
  where, since $X/K_0\geq 5\cdot 10^6>438\,653$, we have used Lemma~\ref{sqf}. The result is concluded by noticing that the total contribution of the above second, third and fourth terms is negative.
\end{proof}

\begin{lem}\label{victoire}
  For $X\ge 4\cdot 10^{7}$, we have
  \begin{equation*}
    \biggl|\sum_{n\le X}\mu(n)\Log^2n\biggr|
    \le( 0.006688\,\Log X-0.0504)X.
  \end{equation*}
\end{lem}
In the above range, Lemma \ref{victoire} improves on \cite[Lemma 7.2]{Ramare*12-2} by almost a factor of $2$.

\begin{proof} We consider different cases.

   \smallskip\noindent{\sl Case 1: suppose that $X\ge 10^{16}$}. 
  Let $K_0=2\cdot 10^{9}$ and $K=10^8$. 
  We deduce from \eqref{light*} that
  \begin{align}\label{light2}
\sum_{n\leq X}\mu(n)\log^2n=-2\gamma+\sum_{\ell\leq X/K_0}\mu(\ell)R_2^*\bigg(\frac{X}{\ell}\bigg)+\sum_{X/K_0<\ell\leq X/K}\mu(\ell)R_2^*\bigg(\frac{X}{\ell}\bigg)&\\
+\sum_{k\leq K}(\Lambda\star\Lambda(k)-\Lambda(k)\log k+2\gamma)M\bigg(\frac{X}{k}\bigg)- R_2^*(K)M\bigg(\frac{X}{K}\bigg).&\nonumber
\end{align}
  for any $K\le K_0$.
  
Now, by Lemma \ref{compR2}, we have that
  $|R_2^*(k)|\le 1.93\sqrt{k}\log k$ when $3\le k\le K_0$. Moreover, as $X/K_0\geq 5\cdot 10^6$ and $\log(K_0/K)<19\,000$, we can estimate the third term in \eqref{light2} with the help of Lemma~\ref{aux2} as
  \begin{align*}
   \left|\sum_{X/K_0<\ell\leq X/K}\mu(\ell)R_2^*\bigg(\frac{X}{\ell}\bigg)\right|
       &\le 1.93\sqrt{X}\log(K_0)
    \sum_{X/K_0<\ell\le X/K}\frac{\mu^2(\ell)}{\sqrt{\ell}}
    \\&\le 1.93\sqrt{X}\log(K_0)
    \frac{12}{\pi^2}\sqrt{\frac{X}{K}}
    \le
    \frac{2.35}{\sqrt{K}}X\log K_0.
  \end{align*}
  Furthermore, by combining Lemma~\ref{BoundR2} and Lemma \ref{sqflog}, we have 
  \begin{align*}
\left|\sum_{\ell\leq
    X/K_0}\mu(\ell)R_2^*\bigg(\frac{X}{\ell}\bigg)\right|
    &\leq 0.011\cdot X\sum_{\ell\leq X/K_0}\frac{\mu^2(\ell)}{\ell}\\
&\leq 0.011\cdot X\bigg(\frac{6}{\pi^2}\log\bigg(\frac{X}{K_0}\bigg)+1.045\bigg),
\end{align*}
since $X/K_0\geq 10^6$.

Finally, by Lemma \ref{aux1}, we may bound the last two terms of \eqref{light2}. All in all, we have
  \begin{align*}
  \frac{1}{X}\biggl|\sum_{n\le X}\mu(n)\Log^2n\biggr|
  \le&
  \ \frac{2\gamma}{X}
  +0.011\bigg(\frac{6}{\pi^2}\log\bigg(\frac{X}{K_0}\bigg)+1.045\bigg)
  \\& +\frac{2.35}{\sqrt{K}}\log(K_0)+0.0762
  \\\le& \ 0.006688\cdot\Log X-0.0504
\end{align*}
     \smallskip\noindent{\sl Case 2: suppose that $X\in[4\cdot 10^7,10^{16})$}. By summation by parts,
  we write
  \begin{equation*}
    \sum_{n\le X}\mu(n)\log^2 n
    =
      M(X)\log^2 X-\int_1^XM(t)\frac{2\log t}{t}dt,
    \end{equation*}
    so that estimation \eqref{boundsM} gives 
    \begin{align*}
      \biggl|
      \sum_{n\le X}\mu(n)\log^2 n
      \biggr|
      &\le
        \sqrt{X}\log^2X+\int_1^X\frac{2\log t}{\sqrt{t}}dt
        \\
        &=\sqrt{X}\log^2X+4\sqrt{X}\log X-8\sqrt{X}+8
      \\&
      \le (\log X+4)\sqrt{X}\log X.
    \end{align*}
    We readily check that $0.006688\log
    X-0.0504\geq (\log X)(\log X+4)/\sqrt{X}$ when
    $X\ge 4\cdot 10^7$.
\end{proof}

\begin{proofbold}[Proof of Theorem~\ref{mainbound}] We consider different cases.

 \smallskip\noindent{\sl Case 1: suppose that $X\geq X_1=10^{16}$}.
   We use Lemma~\ref{summ} with $f=\mu$ and $k=2$ and $X_0=4\cdot 10^7$. Thus, $M(\cdot) = M_0(\mu, \cdot)$
  and, by Lemma \ref{victoire}, we derive
  \begin{align}\label{victory}
    |M(X)|
    &\le
    \frac{0.006688\,\Log X-0.0504}{\Log^{2}X}X
    +\biggl|M(X_0)-\frac{M_2(\mu,X_0)}{\Log^{2}X_0}\biggr|\nonumber
    \\&\qquad\qquad+2\int_{X_0}^X\frac{0.006688\Log t-0.0504}{\Log^{3}t}dt.
  \end{align}
  A computer calculation may handle the above second term. Subsequently, by the use Pari/Gp, we obtain
  \begin{equation}\label{pari}
  \biggl|M(X_0)-\frac{M_2(\mu,X_0)}{\Log^{2}X_0}\biggr|\leq 7.01.
  \end{equation}
  Hence, on combining \eqref{victory} and \eqref{pari}, we derive
   \begin{align}\label{victory_2}
    |M(X)|
   \le&
    \frac{0.006688\,\Log X-0.0504}{\Log^{2}X}X
    +7.01
    +\int_{X_0}^X\bigg(\frac{0.013376}{\Log^{2}t}
    -\frac{0.1008}{\Log^{3}t}\bigg)dt
    \nonumber\\
    \le&
    \frac{0.006688\,\Log X-0.0504}{\Log^{2}X}X
    +7.01\nonumber
    \\&
    +0.013376\bigg(\frac{X}{\log^2 X}-\frac{X_0}{\log^2 X_0}\bigg)
    -(0.1008-0.026752)\int_{X_0}^X\frac{dt}{\Log^{3}t}  
    \end{align}
  where we have used that $(\text{Id}/\log^2 )'=1/\log^2 \cdot(1-2/\log )$. Further, the bound 
  \begin{equation*}
  \frac{X}{\log^3 X}-\frac{X_0}{\log^3 X_0}=\int_{X_0}^X\bigg(\frac{t}{\log^3 t}\bigg)'dt\leq\int_{X_0}^X\frac{dt}{\log^3 t}
  \end{equation*}
  leads to a simplification on \eqref{victory} as
     \begin{align}
    |M(X)|
  \le
    \frac{0.006688\,\Log X-0.0504}{\Log^{2}X}X
    +\bigg(7.01-\frac{0.013376X_0}{\log^2 X_0}+\frac{0.074048X_0}{\log^3 X_0}\bigg)&
\nonumber
    \\
    +\frac{0.013376X}{\log^2 X}    -\frac{0.074048X}{\log^3 X}\qquad\qquad\qquad\qquad\qquad&\nonumber \\
    \le\frac{0.006688}{\Log X}X+\frac{X}{\Log^{2}X}\bigg(-0.0504+0.013376-\frac{0.074048}{\log X_1}-\frac{1186.93\log^2 X_1}{X_1}\bigg)&\nonumber\\
    \leq \frac{0.006688\,\Log X-0.039}{\Log^{2}X}X.\qquad\qquad\qquad\qquad\qquad\qquad\qquad\qquad&\label{victoria}
        \end{align}
 \smallskip\noindent{\sl Case 2: suppose that $X\in[X_2,10^{16}]$}, where $X_2=1.5\cdot 10^7$. Then
 \begin{align*}
 \frac{1}{\sqrt{X}}\frac{0.006688\,\Log X-0.039}{\Log^{2}X}X&\geq\bigg(0.006688-\frac{0.039}{\log(X_2)}\bigg)\frac{\sqrt{X}}{\log X}\\
 &\geq\bigg(0.006688-\frac{0.039}{\log(X_2)}\bigg)\frac{\sqrt{X_2}}{\log X_2}\geq 1.
 \end{align*}
  Therefore, by using~\eqref{boundsM}, the bound \eqref{victoria} is valid in the range $[1.5\cdot 10^7,10^{16}]$.
  
   \smallskip\noindent{\sl Case 3: suppose that $X\leq 1.5\cdot 10^7$}.
We conclude by computer verification by relying on Pari/Gp that the bound \eqref{victoria} holds in the range $[T,1.5\cdot 10^7)$, where $T=1\,798\,118$.
\end{proofbold}

\section{Corollaries}

\begin{proofbold}[Proof of Corollary \ref{mainboundbis}] We consider three cases.

     \smallskip\noindent{\sl Case 1: suppose that $X\geq 10^{14}$}. Let $T=1\,798\,118$ and $X_0=10^{14}$. By a numerical calculation, we obtain that 
     \begin{equation}\label{integral}
     \int_1^T|M(t)|dt\leq 216378740
     \end{equation}
 
     Let $A=216378740$. Now, by Theorem \ref{mainbound}, Lemma \ref{EM} and \eqref{integral}, we derive
  \begin{align*}
    |m(X)|
    &\le
    \frac{0.006688\log X-0.039}{\log^2X}
    +\frac{1}{X^2}\int_{T}^X\frac{(0.006688\Log t-0.039)t}{\Log^2t}dt
    \\&\qquad+\frac{1}{X^2}\int_{1}^{T}|M(t)|dt
    +\frac{8}{3X},
    \\&\le
    \frac{0.006688\log X-0.039}{\log^2X}
    +\frac{1}{X^2}\int_{T}^X\frac{0.006688\,tdt}{\Log t}
    \\&\qquad-\frac{1}{X^2}\int_{T}^X\frac{0.039\,tdt}{\Log^2t}
    +\frac{A}{X^2}+\frac{8}{3X}.
  \end{align*}
Moreover, by using that $(\text{Id}^2/\log )'=\text{Id}/\log \cdot(2-1/\log )$ and then the bound 
  \begin{equation*}
  \frac{X^2}{2\log^2 X}-\frac{T^2}{2\log^2 T}=\frac{1}{2}\int_{T}^X\bigg(\frac{t^2}{\log^2 t}\bigg)'dt\leq\int_{T}^X\frac{tdt}{\log^2 t},
  \end{equation*}
   we derive
    \begin{align*}
    |m(X)|
  &\le
    \frac{0.010032\log X-0.039}{\log^2X}
    -\frac{0.003344}{X^2}\frac{T^2}{\log T}
    \\&\qquad-\frac{1}{X^2}\int_{T}^X\frac{0.035656\,tdt}{\Log^2t}
    +\frac{A}{X^2}+\frac{8}{3X}\\
    &\le
     \frac{0.010032\log X-0.039}{\log^2X}
    -\frac{0.003344}{X^2}\frac{T^2}{\log T}
    \\&\qquad-\frac{0.017828}{X^2}\bigg(\frac{X^2}{\log^2 X}-\frac{T^2}{\log^2 T}\bigg)
    +\frac{A}{X^2}+\frac{8}{3X}.
  \end{align*}
  Now, by rearranging terms, we obtain
    \begin{align}\label{bb_neg}
    |m(X)|
  \le&
     \frac{0.010032}{\log X}+\frac{1}{\log^2X}\bigg(-0.039-0.017828+\frac{A\log^2X_0}{X_0^2}+\frac{8\log^2X_0}{3X_0}\bigg)
    \nonumber\\
    &-\frac{1}{X^2}\bigg(0.003344\frac{T^2}{\log T}
    -0.017828\frac{T^2}{\log^2 T}\bigg)\\
   \label{bb}  \le&
     \frac{0.010032\log X-0.0568}{\log^2 X},
  \end{align}
  where we have used that the expression \eqref{bb_neg} in the above estimation is negative.
  
  \smallskip\noindent{\sl Case 2: suppose that $X\in[X_1,10^{14})$, where $X_1=1.5\cdot 10^7$.}
  By~\eqref{boundsm}, we extend the simplified bound
  $|m(X)|\Log^2X\le 0.01\Log X-0.057$ to any $X\ge X_1 $.
  
  \smallskip\noindent{\sl Case 3: suppose that $X< 1.5\cdot 10^7$}. We verify numerically that the bound \eqref{bb} is valid for any $X\geq 617\,990$, whence the result.
\end{proofbold}

\begin{proofbold}[Proof of Corollary~\ref{mainboundbis2}] Recall Corollary \ref{mainboundbis}. We have that $0.01004\log X-0.056\leq 0.0144\log X-0.1$ when $X\geq 617\,990$. Then, we inspect numerically that the result also holds for $X\in[463\,421,617\,990)$. 
\end{proofbold}

\bibliographystyle{plain}
\bibliography{Local}

\begin{thebibliography}{10}

\bibitem{Broadbent-Kadiri-Lumley-Ng-Wilk*21}
S.~Broadbent, H.~Kadiri, A.~Lumley, N.~Ng, and K.~Wilk.
\newblock Sharper bounds for the {C}hebyshev function {$\theta(x)$}.
\newblock {\em Math. Comp.}, 90(331):2281--2315, 2021.

\bibitem{Buthe*16}
J.~B\"{u}the.
\newblock Estimating {$\pi(x)$} and related functions under partial {RH}
  assumptions.
\newblock {\em Math. Comp.}, 85(301):2483--2498, 2016.

\bibitem{Buthe*18}
J.~B\"{u}the.
\newblock An analytic method for bounding {$\psi(x)$}.
\newblock {\em Math. Comp.}, 87(312):1991--2009, 2018.

\bibitem{Cohen-Dress-ElMarraki*96}
H.~Cohen, F.~Dress, and M.~{El Marraki}.
\newblock Explicit estimates for summatory functions linked to the {M}\"obius
  $\mu$-function.
\newblock {\em Univ. Bordeaux 1}, Pr\'e-publication(96-7), 1996.

\bibitem{Cohen-Dress-ElMarraki*07}
H.~Cohen, F.~Dress, and M.~{El Marraki}.
\newblock Explicit estimates for summatory functions linked to the {M}\"obius
  {$\mu$}-function.
\newblock {\em Funct. Approx. Comment. Math.}, 37:51--63, 2007.

\bibitem{CostaPereira*89}
N.~{Costa Pereira}.
\newblock Elementary estimates for the {C}hebyshev function $\psi({X})$ and for
  the {M}\"obius function ${M}({X})$.
\newblock {\em Acta Arith.}, 52:307--337, 1989.

\bibitem{Daval*24}
F.~Daval.
\newblock Private communication.
\newblock 2024.

\bibitem{Dress*93}
F.~{Dress}.
\newblock Fonction sommatoire de la fonction de {M}\"obius 1. {M}ajorations
  exp\'erimentales.
\newblock {\em Exp. Math.}, 2(2):89--98, 1993.

\bibitem{Dress-ElMarraki*93}
F.~{Dress} and M.~{El Marraki}.
\newblock Fonction sommatoire de la fonction de {M}\"obius 2. {M}ajorations
  asymptotiques \'el\'ementaires.
\newblock {\em Exp. Math.}, 2(2):99--112, 1993.

\bibitem{Helfgott19}
H.~Helfgott.
\newblock {\em The ternary Goldbach conjecture}.
\newblock Book accepted for publication in Ann. of Math. Studies, 2019.

\bibitem{Hurst*18}
G.~Hurst.
\newblock Computations of the {M}ertens function and improved bounds on the
  {M}ertens conjecture.
\newblock {\em Math. Comp.}, 87(310):1013--1028, 2018.

\bibitem{Ramare*12-1}
O.~Ramar\'e.
\newblock Explicit estimates for the summatory function of ${\Lambda}(n)/n$
  from the one of ${\Lambda}(n)$.
\newblock {\em Acta Arith.}, 159(2):113--122, 2013.

\bibitem{Ramare*12-2}
O.~Ramar\'e.
\newblock From explicit estimates for the primes to explicit estimates for the
  {M}oebius function.
\newblock {\em Acta Arith.}, 157(4):365--379, 2013.

\bibitem{Ramare*12-5}
O.~Ramar\'e.
\newblock Explicit estimates on several summatory functions involving the
  {M}oebius function.
\newblock {\em Math. Comp.}, 84(293):1359--1387, 2015.

\bibitem{Ramare*18-9}
O.~Ramar\'{e}.
\newblock Explicit average orders: {N}ews and {P}roblems.
\newblock In {\em Number theory week 2017}, volume 118 of {\em Banach Center
  Publ.}, pages 153--176. Polish Acad. Sci. Inst. Math., Warsaw, 2019.

\bibitem{Ramare-Rumely*96}
O.~{Ramar\'e} and R.~Rumely.
\newblock Primes in arithmetic progressions.
\newblock {\em Math. Comp.}, 65:397--425, 1996.

\bibitem{Rosser-Schoenfeld*62}
J.B. Rosser and L.~Schoenfeld.
\newblock Approximate formulas for some functions of prime numbers.
\newblock {\em Illinois J. Math.}, 6:64--94, 1962.

\bibitem{Rosser-Schoenfeld*75}
J.B. {Rosser} and L.~{Schoenfeld}.
\newblock Sharper bounds for the {C}hebyshev {F}unctions $\vartheta(x)$ and
  $\psi(x)$.
\newblock {\em Math. Comp.}, 29(129):243--269, 1975.

\bibitem{Vanlalnagaia*15-1}
R.~Vanlalngaia.
\newblock Explicit mertens sums.
\newblock {\em Integers}, 17:18pp, 2017.
\newblock A11.

\end{thebibliography}

\Addresses
\end{document}